\documentclass{proc-l}

\usepackage[colorlinks,linkcolor=black,citecolor=black]{hyperref}
\newtheorem{theorem}{Theorem}[section]

\newtheorem{conjecture}[theorem]{Conjecture}

\theoremstyle{definition}
\newtheorem{definition}[theorem]{Definition}

\theoremstyle{remark}

\numberwithin{equation}{section}

\makeatletter
\@namedef{subjclassname@2020}{\textup{2020} Mathematics Subject Classification}
\makeatother

\begin{document}

\title[On two conjectures in spectral hypergraph theory]{On two conjectures in spectral hypergraph theory}


\author[Y.-N. Zheng]{Ya-Nan Zheng}
\address{School of Mathematics and Statistics, Henan Normal University, Xinxiang 453007, People's Republic of China}
\email{zh\_yn@tju.edu.cn}

\thanks{This work was supported by National Natural Science Foundation of China (Grant No. 12601659) and Natural Science Foundation of Henan Province (Grant No. 252300421785)}



\date{}

\dedicatory{}


\begin{abstract}
For a connected uniform hypergraph \(H\), we prove two conjectures concerning the algebraic multiplicities of the spectral radius of its adjacency tensor and the zero eigenvalue of its Laplacian tensor. In particular, we establish that the algebraic multiplicity of each of these eigenvalues is equal to the cardinality of the corresponding projective eigenvariety.
\end{abstract}

\maketitle

\section{Introduction}

A tensor (or hypermatrix) $\mathcal{A} = (a_{i_1 \cdots i_k})$ of order $k$ and dimension $n$ over the complex field $\mathbb{C}$ is a multidimensional array whose entries $a_{i_1 \cdots i_k} \in \mathbb{C}$, where $i_j\in[n]$ for all $j\in[k]$. Associate $\mathcal{A}$ with a directed graph $D(\mathcal{A})$ on vertex set $[n]$ such that $(i, j)$ is an arc of $D(\mathcal{A})$ if and only if there exists a nonzero entry $a_{i i_2 \cdots i_k}$ such that $j \in \{i_2, \ldots, i_k\}$. Then $\mathcal{A}$ is called \emph{weakly irreducible} if $D(\mathcal{A})$ is strongly connected; otherwise, it is called \emph{weakly reducible}. We say $\mathcal{A}$ is \emph{nonnegative} if all its entries are nonnegative, and $\mathcal{A}$ is \emph{symmetric} if all entries $a_{i_1 \cdots i_k}$ are invariant under any permutation of their indices. An order $k$ dimension $n$ tensor $\mathcal{I} = (i_{i_1 \cdots i_k})$ is called an \emph{identity tensor} if $i_{i_1 \cdots i_k} = 1$ for $i_1 = \cdots = i_k$, and $i_{i_1 \cdots i_k} = 0$ otherwise.

For $\mathbf{x} = (x_1, \ldots, x_n)^{\rm T} \in \mathbb{C}^n$, the vector $\mathcal{A}\mathbf{x}^{k-1} \in \mathbb{C}^n$ is defined componentwise by
\[
(\mathcal{A} \mathbf{x}^{k-1})_i = \sum_{i_2, \ldots, i_k = 1}^{n} a_{i i_2 \cdots i_k}x_{i_2} \cdots x_{i_k}.
\]

Lim \cite{Lim05} and Qi \cite{Qi05} proposed eigenvalues for higher order tensors independently.
\begin{definition} [\cite{Lim05, Qi05}]
Let $\mathcal{A}$ be a $k$-th order $n$-dimensional tensor. If there exists a number $\lambda \in \mathbb{C}$ and a vector $\mathbf{x} \in \mathbb{C}^n \setminus \{\mathbf{0}\}$ satisfy the following homogeneous polynomial system
\[
\mathcal{A} \mathbf{x}^{k-1} = \lambda \mathbf{x}^{[k-1]},
\]
where $\mathbf{x}^{[k-1]} = (x_1^{k-1}, \ldots, x_n^{k-1})^{\rm T}$. Then $\lambda$ is called an \emph{eigenvalue} of $\mathcal{A}$ and $\mathbf{x}$ is called an \emph{eigenvector} of $\mathcal{A}$ corresponding to $\lambda$.
\end{definition}

Qi \cite{Qi05} also introduced the determinant and characteristic polynomial of a tensor, both of which are based on the resultant of a homogeneous polynomial system. The \emph{determinant} of $\mathcal{A}$, denoted by $\det(\mathcal{A})$, is the resultant of $\mathcal{A} \mathbf{x}^{k-1} = \mathbf{0}$, i.e., $\det(\mathcal{A}) = {\rm Res}(\mathcal{A} \mathbf{x}^{k-1})$. The characteristic polynomial $\varphi_{\mathcal{A}}(\lambda)$ of $\mathcal{A}$ is defined to be $\det(\mathcal{\lambda \mathcal{I} - \mathcal{A}}) \in \mathbb{C}[\lambda]$. It follows that $\lambda \in \mathbb{C}$ is an eigenvalue of $\mathcal{A}$ if and only if it is a root of $\varphi_{\mathcal{A}}(\lambda)$. Following the notation of \cite{Fan26, HY16}, we denote the \emph{algebraic multiplicity} of an eigenvalue $\lambda$ of $\mathcal{A}$ by $\operatorname{am}(\lambda, \mathcal{A})$, defined as the multiplicity of $\lambda$ as a root of the characteristic polynomial $\varphi_{\mathcal{A}}(\lambda)$. The \emph{spectrum} of $\mathcal{A}$, written ${\rm Spec}(\mathcal{A})$, is the multiset of roots of $\varphi_{\mathcal{A}}(\lambda)$. The \emph{spectral radius} of $\mathcal{A}$, denoted $\rho(\mathcal{A})$, is the largest modulus of the elements in ${\rm Spec}(\mathcal{A})$.

For a given eigenvalue $\lambda$ of $\mathcal{A}$, the \emph{projective eigenvariety} of $\mathcal{A}$ associated with $\lambda$ is defined to be the projective variety \cite{Fan26, FBH19, FHB22, FTL22, FWBWLZ19, HY16}
\[
\mathbb{V}_{\lambda}(\mathcal{A}) = \{\mathbf{x} \in \mathbb{P}^{n-1}: \mathcal{A} \mathbf{x}^{k-1} = \lambda \mathbf{x}^{[k-1]}\},
\]
where $\mathbb{P}^{n-1}$ is the complex projective space over $\mathbb{C}$ of dimension $n-1$.

Friedland et al. \cite{FGH13}, Yang and Yang \cite{YY11} generalized the famous Perron-Frobenius theorem from irreducible nonnegative matrices to weakly irreducible nonnegative tensors.

\begin{theorem}[\cite{FGH13, YY11}] \label{PFthm}
Let $\mathcal A$ be a weakly irreducible nonnegative tensor. Then there exists a positive eigenvector associated with $\rho(\mathcal{A})$, unique up to positive scalar multiplication, called the Perron vector of $\mathcal{A}$. Moreover, if $\mathbf{x}$ is an eigenvector corresponding to an eigenvalue of modulus $\rho(\mathcal{A})$, then $|\mathbf{x}|$ is a positive Perron eigenvector.
\end{theorem}

A \emph{hypergraph} $H=(V(H),E(H))$ consists of a vertex set $V(H)$ and an edge set $E(H) \subseteq \mathcal{P}(V(H))$, where $\mathcal{P}(V(H))$ denotes the power set of $V(H)$. The \emph{degree} of a vertex $v$ of $H$, denoted by $d_v$, is the number of edges of $H$ containing $v$. If for each edge $e \in E(H)$, $|e|=k$, then $H$ is called \emph{$k$-uniform}. For a $k$-uniform hypergraph $H$ on $V(H) = [n]$, the \emph{adjacency tensor} \cite{CD12} of $H$ is defined as the order $k$ dimension $n$ tensor $\mathcal{A}(H)$ with
\[
a_{i_{1} \cdots i_{k}} = \begin{cases}
    	\frac{1}{(k-1)!},  & \textrm{ if } \{i_{1}, \dots, i_{k}\} \in E(H),  \\
    	~~0, & \textrm{ otherwise}.
\end{cases}
\]
The spectral radius of $H$ is defined as that of the adjacency tensor $\mathcal{A}(H)$. The \emph{Laplacian tensor} of $H$ is defined as $\mathcal{L}(H) = \mathcal{D}(H) - \mathcal{A}(H)$ and the \emph{signless Laplacian tensor} of $H$ is defined as $\mathcal{Q}(H) = \mathcal{D}(H) + \mathcal{A}(H)$ \cite{Qi14}, where $\mathcal{D}(H)$ is a $k$-th order and $n$-dimensional diagonal tensor such that $d_{i \cdots i}=d_i$, the degree of the vertex $i$ for each $i \in V(H)$.

Two distinct vertices $i$ and $j$ are said to be connected if there is a sequence of edges $(e_{1}, \ldots, e_{m})$ such that $i \in e_{1}, j \in e_{m}$ and $e_{r} \cap e_{r+1} \neq \emptyset$ for all $r \in [m-1]$. A hypergraph is called \emph{connected} if every pair of vertices of $H$ is connected. A \emph{hypertree} is a connected hypergraph without cycles. It is clear that $\mathcal{A}(H), \mathcal{L}(H)$, and $\mathcal{Q}(H)$ are all symmetric tensors. Moreover, as shown in \cite{FGH13, YY11}, each of them is weakly irreducible if and only if $H$ is connected.

Recently, Fan \cite{Fan26} proved that the equalities $\operatorname{am}(\rho, \mathcal{A}(H)) = |\mathbb{V}_{\rho}(\mathcal{A}(H))|$ and $\operatorname{am}(0, \mathcal{L}(H)) = \operatorname{am}(\rho, \mathcal{A}(H))$ hold for certain classes of connected hypergraphs, and conjectured that they remain valid for all connected uniform hypergraphs.

\begin{conjecture} [\cite{Fan26}] \label{conj1}
Let $H$ be a connected $k$-uniform hypergraph with spectral radius $\rho$. Then
\[
{\rm am}(\rho, \mathcal{A}(H)) = |\mathbb{V}_{\rho}(\mathcal{A}(H))|.
\]
\end{conjecture}

\begin{conjecture} [\cite{Fan26}] \label{conj2}
Let $H$ be a connected $k$-uniform hypergraph with spectral radius $\rho$. Then
\[
{\rm am}(0, \mathcal{L}(H)) = {\rm am}(\rho, \mathcal{A}(H)).
\]
\end{conjecture}

In this paper, we establish these two conjectures for connected uniform hypergraphs. In preparation for the proofs, we will review some preliminaries that are essential in the subsequent discussions.

\section{Preliminaries}

For a positive integer $n$, an $n$-tuple $\alpha = (\alpha_{1}, \ldots, \alpha_{n})$ of nonnegative integers, and an $n$-tuple $\mathbf{x} = (x_{1}, \ldots, x_n)$ of indeterminate variables, denote by $\mathbf{x}^\alpha$ the monomial $\prod_{i=1}^{n}x_{i}^{\alpha_{i}}$. The total degree of the monomial $\mathbf{x}^{\alpha}$, denoted by $|\alpha|$, is the sum of its exponents, that is, $|\alpha| = \sum_{i=1}^n \alpha_i$. Let $F(x_{1}, \ldots, x_{n}) \in \mathbb{C}[x_{1}, \ldots, x_n]$, when no ambiguity can arise, we abbreviate $F(x_1, \ldots, x_n)$ by $F$.

\begin{theorem} [\cite{CLO98}]\label{resultant}
Fix degrees $d_1, \ldots, d_n$. For $i \in [n]$, consider all monomials $\mathbf{x}^{\alpha}$ of total degree $d_i$ in $x_1, \ldots, x_n$. For each such monomial, define a variable $u_{i,\alpha}$. Then there is a unique polynomial {\rm Res} $\in \mathbb{Z}[\{u_{i,\alpha}\}]$ with the following three properties:
\begin{enumerate}
  \item[{\rm (1)}]  If $F_1, \ldots, F_n \in \mathbb{C}[x_1,\ldots, x_n]$ are homogeneous of degrees $d_1, \ldots, d_n$, respectively, then the polynomials have a nontrivial common root in $\mathbb{C}^n$ exactly when {\rm Res}$(F_1, \ldots, F_n) = 0$.
  \item[{\rm (2)}] {\rm Res}$(x_1^{d_1}, \ldots, x_n^{d_n}) = 1$.
  \item[{\rm (3)}] {\rm Res} is irreducible, even in $\mathbb{C}[\{u_{i,\alpha}\}]$.
\end{enumerate}
\end{theorem}

$\textrm{Res}(F_1, \ldots, F_n)$ is called the resultant of $F_1, \ldots, F_n$ and it is interpreted to mean substituting the coefficient of $\mathbf{x}^\alpha$ in $F_i$ for the variable $u_{i, \alpha}$ in {\rm Res}.

Let $\mathbb{K}$ be an algebraically closed field, and let $F_{1}, \ldots, F_{n} \in \mathbb{K}[x_{1}, \ldots, x_{n}]$ be homogeneous polynomials of degrees $d_{1}, \ldots, d_{n}$, respectively. Define
\begin{align*}
f_i(x_{1}, \ldots, x_{n-1}) &= F_{i}(x_{1}, \ldots, x_{n-1}, 1), \\
\overline{F}_{i}(x_{1}, \ldots, x_{n-1}) &= F_{i}(x_{1}, \ldots, x_{n-1}, 0).
\end{align*}
Note that $\overline{F}_{1}, \ldots, \overline{F}_{n-1}$ are homogeneous in $\mathbb{K}[x_{1}, \ldots, x_{n-1}]$. The following Poisson formula in {\cite{CLO98, GKZ94}} gives a recursive method for calculating resultants.

\begin{theorem} [\cite{CLO98}] \label{Poisson}
Let $F_1, \ldots, F_n \in \mathbb{K}[x_1, \ldots, x_n]$. If ${\rm Res}(\overline{F}_1, \ldots, \overline{F}_{n-1}) \neq 0$, then the quotient ring $A = \mathbb{K}[x_1, \ldots, x_{n-1}] / \langle f_1, \ldots, f_{n-1} \rangle$ is a finite-dimensional $\mathbb{K}$-vector space of dimension $d_1 \cdots d_{n-1}$. Moreover,
\[
{\rm Res}(F_1, \ldots, F_n) = {\rm Res}(\overline{F}_1, \ldots, \overline{F}_{n-1})^{d_n} \det(m_{f_n}: A \rightarrow A),
\]
where $m_{f_n}: A \rightarrow A$ is the $\mathbb{K}$-linear map given by multiplication by $f_n$.
\end{theorem}

We next recall the notion of local intersection multiplicity of a point in an algebraic variety (see \cite{CLO97, CLO98, Fan26}). For a point $p=(p_1, \ldots, p_n) \in \mathbb{K}^n$, let
\[
M_p = \langle x_1-p_1, \ldots, x_n-p_n \rangle
\]
denote the maximal ideal of $\mathbb{K}[x_1, \ldots, x_n]$ corresponding to $p$. The localization of $\mathbb{K}[x_1, \ldots, x_n]$ at $M_p$ is defined as
\[
\mathbb{K}[x_1, \ldots, x_n]_{M_p} = \left\{\frac{f}{g}: f, g\in \mathbb{K}[x_1, \ldots, x_n],\ g(p) \neq 0 \right\},
\]
viewed as a subring of the field $\mathbb{K}(x_1,\ldots,x_n)$ of rational functions over $\mathbb{K}$. It is well known that this localization is a local ring.

Let $J$ be a zero-dimensional ideal in $\mathbb{K}[x_1, \ldots, x_n]$ and $\mathcal{V}(J)$ denote the zero set of $J$. For a point $p = (p_{1}, \ldots, p_{n}) \in \mathcal{V}(J)$, write
\[
J_p := J \mathbb{K}[x_1,\ldots,x_n]_{M_p}
\]
for the ideal of the local ring generated by $J$. The local intersection multiplicity at $p$ with respect to $J$ is defined as
\[
m(p)=\dim_{\mathbb{K}} \left(\mathbb{K}[x_1, \ldots, x_n]_{M_p} / J_p \right),
\]
where the dimension is taken as a $\mathbb{K}$-vector space. If $m(p) = 1$, then $p$ is called a \emph{simple zero} in $\mathcal{V}(J)$.

Fan \cite[Corollary 5.2]{Fan26} established the following equivalent characterization of a simple zero in terms of the Jacobian matrix.

\begin{theorem} [\cite{Fan26}] \label{Jac}
Let $p \in \mathbb{K}^n$ be an isolated zero of $f_1, \ldots, f_t \in \mathbb{K}[x_1, \ldots, x_n]$, where $t \ge n$.
Then $p$ is a simple zero if and only if the Jacobian matrix
\[
J(f_1, \ldots, f_t) = \left(\frac{\partial{f_i}}{\partial{x_j}}\right)_{t \times n}
\]
has rank $n$ at $p$.
\end{theorem}

For an isolated point of a projective variety defined by a homogeneous ideal, its local intersection multiplicity is defined as the local intersection multiplicity of the corresponding dehomogenized ideal on any affine chart containing the point. This number is independent of the chosen affine chart.

For a parameter-dependent system, we use the following local version of the Poisson formula \cite[Lemma 3.2]{Fan26}. Let $R$ be a discrete valuation ring with uniformizer $t$, fraction field $K$, and algebraically closed residue field $\mathbf{k}=R/(t)$.  Every $0 \neq a \in K$ can be written uniquely as
\[
a = t^{v_t(a)}u, \quad u \in R^{\times},
\]
which defines the normalized $t$-adic valuation $v_t: K^{\times} \rightarrow \mathbb Z$.

Let
\[
F_1, \ldots, F_{l+1} \in R[x_1, \ldots, x_l, z]
\]
be homogeneous polynomials of degrees $d_1,\ldots,d_{l+1}$. For $i \in [l+1]$, set
\[
\overline{F}_i(x_1, \ldots, x_l) = F_i(x_1, \ldots, x_l, 0), ~~f_i(x_1, \ldots, x_l) = F_i(x_1, \ldots, x_l, 1).
\]
Suppose that $r = {\rm Res}(\overline{F}_1, \ldots, \overline{F}_l) \in R^{\times}$ and that ${\rm Res}(F_1, \ldots, F_{l+1}) \neq 0$.  Then
\[
Q = R[x_1,\ldots,x_l] / \langle f_1, \ldots, f_l \rangle
\]
is a finite free $R$-module of rank $d_1 \cdots d_l$.  Let $M$ be multiplication by $f_{l+1}$ on $Q$, then
\[
v_t \! \left({\rm Res}(F_1, \ldots, F_{l+1}) \right) = v_t(\det M).
\]

We shall also use the standard decomposition theorem for finite-dimensional commutative algebras. In particular, every finite-dimensional commutative algebra over a field is Artinian and can be expressed as a finite direct product of local Artinian algebras (see \cite[Chapter 8]{AM69}).

We next recall a standard result on lifting idempotents over complete rings \cite [Corollary 7.5]{Eisenbud95}.

\begin{theorem} [\cite{Eisenbud95}] \label{thm:idempotent-lifting}
Let $R$ be a commutative Noetherian ring complete with respect to an ideal $\mathfrak{m}$, and let $A$ be an $R$-algebra that is finite as an $R$-module. Then any finite family of pairwise orthogonal idempotents in $A/\mathfrak{m}A$ can be lifted to a family of pairwise orthogonal idempotents in $A$. Moreover, if $A$ is commutative, then the lifting of each idempotent is unique.
\end{theorem}

Finally, we shall use both the holomorphic and formal implicit function theorems. The holomorphic implicit function theorem (see \cite[p.34]{FG02}) ensures that if a system of holomorphic equations
\[
f_1(\lambda, \mathbf{x}) = \cdots = f_l(\lambda, \mathbf{x}) = 0
\]
has a common solution $(\lambda_0, \mathbf{q})$ at which the Jacobian matrix with respect to $x_1, \ldots, x_l$ is nonsingular, then there exist neighborhoods $U$ of $\lambda_0$ and $W$ of
$\mathbf{q}$, and a unique holomorphic map $\mathbf{x}: U\to W$ such that
\[
f_i(\lambda, \mathbf{x}(\lambda)) = 0, \quad i \in [l],
\]
and $\mathbf{x}(\lambda_0) = \mathbf{q}$.

We shall also use the formal implicit function theorem (see \cite[A.4.7]{Hazewinkel78}). In particular, if
\[
f_i \in \mathbb{C}[[t, x_1-q_1, \ldots, x_l-q_l]], \quad i \in [l],
\]
satisfy
\[
f_i(0, \mathbf{q}) = 0, \quad i \in [l],
\]
and
\[
\det \left(\frac{\partial(f_1,\ldots,f_l)}{\partial(x_1,\ldots,x_l)}(0,\mathbf{q}) \right)\neq 0,
\]
then there exists a unique $\mathbf{X}(t) \in \mathbb{C}[[t]]^l$ such that
\[
\mathbf{X}(0) = \mathbf{q}, \quad f_i(t, \mathbf{X}(t)) = 0, \quad i \in [l].
\]

\section{Proof of Conjecture \ref{conj1}} \label{sec3}

Let $H$ be a connected $k$-uniform hypergraph on $[n]$ with spectral radius $\rho$. For each $i \in [n]$, let $E_i$ denote the set of edges containing $i$. The eigenvalue equations are

\begin{equation} \label{charequation}
F_i(\lambda, \mathbf{x}) = \lambda x_i^{k-1} - \sum_{e \in E_{i}} \mathbf{x}^{e \setminus i} = 0, \quad i \in [n].
\end{equation}

Khan and Fan \cite[Corollary 3.5]{KF15} compared the spectral radius of a connected hypergraph with its proper subhypergraphs.

\begin{theorem} [\cite{KF15}] \label{compspec}
Suppose $H$ is a connected $k$-uniform hypergraph and $H'$ is a proper subhypergraph of $H$. Then
\[
\rho(\mathcal{A}(H')) < \rho(\mathcal{A}(H)).
\]
\end{theorem}

Fan \cite[Theorem~1.8]{Fan26} also investigated the local intersection multiplicity of the eigenvectors in the projective variety $\mathbb{V}_{\lambda}(\mathcal{A}(H))$ with $|\lambda|=\rho(\mathcal{A}(H))$.

\begin{theorem}[\cite{Fan26}] \label{asimple}
Let $H$ be a connected $k$-uniform hypergraph. Then for any eigenvalue $\lambda$ of $\mathcal{A}(H)$ with $|\lambda| = \rho(\mathcal{A}(H))$, each point $p \in \mathbb{V}_{\lambda}(\mathcal{A}(H))$ has local intersection multiplicity $1$.
\end{theorem}

\begin{theorem}\label{thm1}
Let $H$ be a connected $k$-uniform hypergraph with spectral radius $\rho$. Then
\[
\operatorname{am}(\rho, \mathcal{A}(H))
= \left|\mathbb{V}_{\rho}(\mathcal{A}(H))\right|.
\]
\end{theorem}

\begin{proof}
For $i \in [n]$, let $F_i(\lambda, \mathbf{x})$ be defined as \eqref{charequation} and $\mathbf{x}'=(x_1, \ldots, x_{n-1})^{\rm T}$. Denote
\[
\overline{F}_i(\lambda, \mathbf{x}') = F_i(\lambda, x_1, \ldots, x_{n-1}, 0), \quad f_i(\lambda, \mathbf{x}') = F_i(\lambda, x_1, \ldots, x_{n-1}, 1).
\]
Let $r(\lambda) = \operatorname{Res}(\overline{F}_1, \ldots, \overline{F}_{n-1})$, by Theorem~\ref{compspec}, $r(\rho) \neq 0$.

Fix a projective point $P \in \mathbb{V}_{\rho}(\mathcal{A}(H))$. By Theorem~\ref{PFthm}, all coordinates of $P$ are nonzero. Thus $P$ has a unique representative
\[
\mathbf{p} = (p_1, \ldots, p_n)^{\rm T} ~~\text{with}~~ p_n=1.
\]
Let $\mathbf{p}' = (p_1, \ldots, p_{n-1})^{\rm T}$, and set
\[
J(\mathbf{p}) = \left(\frac{\partial F_i}{\partial x_j}(\rho, \mathbf{p})\right)_{n \times n}.
\]

On the affine chart $x_n = 1$, the point $P$ corresponds to $\mathbf{p}'$, a common zero of $f_1(\rho, \mathbf{x}'), \ldots, f_n(\rho, \mathbf{x}')$. Since $p_{n} \neq 0$, dehomogenization on the affine chart $x_{n} \neq 0$ preserves the local intersection multiplicity. Theorem~\ref{asimple} shows that its local intersection multiplicity is one. Hence $\mathbf{p}'$ is a simple zero, and Theorem~\ref{Jac} gives
\[
D_{\mathbf{p}} = \left(\frac{\partial f_i}{\partial x_j}(\rho, \mathbf{p}')\right)_{n \times (n-1)}, \quad \operatorname{rank}D_{\mathbf{p}} = n-1.
\]
In particular, this matrix consists precisely of the first $n-1$ columns of $J(\mathbf{p})$.

On the other hand, since each $F_i(\rho, \mathbf{x})$ is homogeneous of degree $k-1$, Euler's identity yields
\[
J(\mathbf{p}) \mathbf{p} = (k-1)
\begin{pmatrix}
F_1(\rho, \mathbf{p}) \\
\vdots \\
F_n(\rho, \mathbf{p})
\end{pmatrix}
= \mathbf{0}.
\]
It follows that
\[
\operatorname{rank}J(\mathbf{p}) = n-1, \quad \ker J(\mathbf{p}) = \operatorname{span}\{\mathbf{p}\}.
\]
Note that $J(\mathbf{p})$ is symmetric and its adjugate has rank one. The column space and row space of the adjugate are spanned by $\mathbf{p}$ and $\mathbf{p}^{\rm T}$, respectively. Therefore, for some $c_{\mathbf{p}} \neq 0$,
\begin{equation} \label{adjmatrix}
\operatorname{adj}J(\mathbf{p}) = c_{\mathbf{p}}\mathbf{p} \mathbf{p}^{\rm T}.
\end{equation}

Define
\[
\Delta_{\mathbf{p}} =\left(\frac{\partial f_i}{\partial x_j}(\rho, \mathbf{p}')\right)_{(n-1) \times (n-1)}.
\]
This is the leading principal $(n-1) \times (n-1)$ submatrix of $J(\mathbf{p})$. Hence
\begin{equation} \label{detdelta}
\det \Delta_{\mathbf{p}} = \operatorname{adj}J(\mathbf{p})_{nn} = c_{\mathbf{p}}p_n^2 = c_{\mathbf{p}} \neq 0.
\end{equation}

By the holomorphic implicit function theorem, there exist neighborhoods $U$ of $\rho$ and $W$ of $\mathbf{p}'$, and a unique holomorphic map $\mathbf{x}_{\mathbf{p}}: U\to W$ such that
\[
f_i(\lambda, \mathbf{x}_{\mathbf{p}}(\lambda)) = 0, \quad i \in [n-1]
\]
and $\mathbf{x}_{\mathbf{p}}(\rho) = \mathbf{p}'$. Let
\[
g_{\mathbf{p}}(\lambda) = f_n(\lambda, \mathbf{x}_{\mathbf{p}}(\lambda)),
\]
then we claim that $g_{\mathbf{p}}(\lambda)$ has a simple zero at $\lambda = \rho$.

Consider the augmented Jacobian
\[
J_{\mathrm{aug}}(\mathbf{p}) = \left. \frac{\partial(f_1, \ldots, f_n)}{\partial(x_1, \ldots, x_{n-1}, \lambda)} \right|_{(\rho, \mathbf{p}')}.
\]
Since
\[
\frac{\partial f_i}{\partial \lambda}(\rho, \mathbf{p}') = \frac{\partial F_i}{\partial \lambda}(\rho, \mathbf{p}) = p_i^{k-1}, \quad i \in [n],
\]
then the last column of $J_{\mathrm{aug}}(\mathbf{p})$ is $(p_1^{k-1}, \ldots, p_n^{k-1})^{\rm T}$. Expanding the determinant along this column and using \eqref{adjmatrix}, we obtain
\[
\begin{aligned}
\det J_{\mathrm{aug}}(\mathbf{p})
&= \sum_{i=1}^{n} \operatorname{adj}J(\mathbf{p})_{ni}p_i^{k-1} \\
&= c_{\mathbf{p}}p_n \sum_{i=1}^{n}p_i^k \\
&= c_{\mathbf{p}} \sum_{i=1}^{n}p_i^k,
\end{aligned}
\]
where we have used $p_n = 1$.

Let
\[
\mathbf{a} = \left(
\frac{\partial f_n}{\partial x_1}(\rho, \mathbf{p}'), \ldots, \frac{\partial f_n}{\partial x_{n-1}}(\rho, \mathbf{p}')
\right)^{\rm T},
\quad
\mathbf{b} = \left(
\frac{\partial f_1}{\partial \lambda}(\rho, \mathbf{p}'),\ldots, \frac{\partial f_{n-1}}{\partial \lambda}(\rho, \mathbf{p}')
\right)^{\rm T},
\]
and $d = \frac{\partial f_n}{\partial \lambda}(\rho, \mathbf{p}')$. Then the augmented Jacobian matrix has the following block form
\[
J_{\mathrm{aug}}(\mathbf{p}) =
\begin{pmatrix}
\Delta_{\mathbf{p}} & \mathbf{b} \\
\mathbf{a}^{\rm T} & d
\end{pmatrix}.
\]
Differentiating the first $n-1$ equations
\[
f_i \left(\lambda, \mathbf x_{\mathbf{p}}(\lambda) \right) = 0, \quad i \in [n-1],
\]
with respect to $\lambda$ and evaluating at $\lambda = \rho$ gives
the linear system
\[
\Delta_{\mathbf{p}} \dot{\mathbf{x}}_{\mathbf{p}} + \mathbf{b} = \mathbf{0}, \quad \dot{\mathbf{x}}_{\mathbf{p}} := \left.
\frac{\mathrm{d} \mathbf{x}_{\mathbf{p}}(\lambda)}{\mathrm{d} \lambda}
\right|_{\lambda = \rho}.
\]
Since $\Delta_{\mathbf{p}}$ is nonsingular, it follows that $\dot{\mathbf{x}}_{\mathbf{p}} = -\Delta_{\mathbf{p}}^{-1} \mathbf{b}$.

Therefore, by the Schur complement formula and
\eqref{detdelta},
\begin{equation} \label{partequ}
\begin{aligned}
g_{\mathbf{p}}'(\rho)
&= d + \mathbf{a}^{\rm T} \dot{\mathbf{x}}_{\mathbf{p}}\\
&= d - \mathbf{a}^{\rm T} \Delta_{\mathbf{p}}^{-1}\mathbf{b}\\
&= \frac{\det J_{\mathrm{aug}}(\mathbf{p})} {\det\Delta_{\mathbf{p}}} = \sum_{i=1}^{n}p_i^k.
\end{aligned}
\end{equation}

We now prove that the sum on the right-hand side is nonzero. Write $p_i = |p_i| e^{\mathrm{i} \theta_i}$. By Theorem~\ref{PFthm}, $|\mathbf{p}|$ is a positive Perron vector. For each $i \in [n]$, let $E_i = \{e \in E(H): i \in e\}$. The eigenvalue equations and the triangle inequality give
\[
\begin{aligned}
\rho|p_i|^{k-1}
&=\left|
\sum_{e \in E_i} \prod_{j \in e \setminus \{i\}}p_j
\right| \\
&\le \sum_{e \in E_i} \prod_{j \in e \setminus \{i\}}|p_j|\\
&= \rho|p_i|^{k-1}.
\end{aligned}
\]
Thus equality holds in the triangle inequality. Since all coordinates of $\mathbf{p}$ are nonzero and $\rho > 0$, every summand has argument $(k-1) \theta_i$ modulo $2\pi$.
Consequently, for every edge $e \in E(H)$ and every $i \in e$,
\[
\arg \left(\prod_{j \in e}p_j\right) \equiv k \theta_i \pmod {2\pi}.
\]
Hence $e^{\mathrm{i} k \theta_i} = e^{\mathrm{i} k \theta_j}$ whenever $i$ and $j$ belong to a common edge. The connectedness of $H$ propagates these
equalities to all vertices. Since $p_n = 1$, we conclude that $ e^{\mathrm{i} k \theta_i} = 1$ for every $i \in [n]$. Therefore
\begin{equation} \label{sumlarge}
\sum_{i=1}^{n}p_i^k = \sum_{i=1}^{n}|p_i|^k > 0.
\end{equation}
As $P \in \mathbb{V}_{\rho}(\mathcal{A}(H)), g_{\mathbf{p}}(\rho) = 0$. Equations \eqref{partequ} and \eqref{sumlarge} show that
$g_{\mathbf{p}}'(\rho) \neq 0$. Thus $g_{\mathbf{p}}(\lambda)$ has a simple zero at $\lambda = \rho$.

Let
\[
R:= \mathbb{C}[\lambda]_{(\lambda-\rho)} =\left\{\frac{f(\lambda)}{g(\lambda)}: f, g \in \mathbb{C}[\lambda],\ g(\rho) \neq 0 \right\}.
\]
and
\[
Q = R[x_1,\ldots,x_{n-1}] / \langle f_1, \ldots, f_{n-1} \rangle.
\]
Then $R$ is a discrete valuation ring with uniformizer $t = \lambda-\rho$, maximal ideal $(t)$, residue field $R/(t) \cong \mathbb{C}$, and fraction field $\mathbb{C}(\lambda)$. Its normalized valuation $v_t$ is the order of vanishing at $\lambda = \rho$.

Let $M: Q\to Q$ be the $R$-linear map given by multiplication by $f_n$. Then $r(\lambda)$ is a unit in $R$ and $\operatorname{Res}(F_1, \ldots, F_n) = \varphi_{\mathcal{A}(H)}(\lambda)$ is a nonzero polynomial. Hence the hypotheses of \cite[Lemma 3.2]{Fan26} are satisfied, and
\begin{equation} \label{amequal}
\operatorname{am}(\rho, \mathcal{A}(H)) = v_t(\det M).
\end{equation}

We now compute $v_t(\det M)$. Let
\[
\widehat{R} = \mathbb C[[t]],\quad \widehat{Q} = Q \otimes_{R} \widehat{R}, \quad \widehat{M} = M \otimes_{R} \widehat{R}.
\]
Since $t = \lambda-\rho$, the natural embedding $R \hookrightarrow \widehat{R}$ is given by expanding each element of $R$ as a formal power series in $t$. By extending scalars from $R$ to
$\widehat{R}$, we obtain that
\[
R[x_{1}, \ldots, x_{n-1}] \otimes_{R} \widehat{R} \cong \widehat{R}[x_{1}, \ldots, x_{n-1}].
\]
The quotient presentation of $Q$ gives the right-exact sequence
\[
\langle f_{1}, \ldots, f_{n-1} \rangle \rightarrow R[x_{1}, \ldots, x_{n-1}] \rightarrow Q \rightarrow 0.
\]
Tensoring this right-exact sequence over $R$ with $\widehat{R}$ yields
\[
\begin{aligned}
\widehat{Q}
&\cong \frac{\widehat{R}[x_{1}, \ldots, x_{n-1}]} {\langle f_{1}, \ldots, f_{n-1} \rangle \widehat{R}[x_{1}, \ldots, x_{n-1}]} \\
&\cong \frac{\mathbb{C}[[t]][x_1, \ldots, x_{n-1}]}
{\langle f_1(t+\rho, \mathbf{x}'), \ldots, f_{n-1}(t+\rho, \mathbf{x}') \rangle}.
\end{aligned}
\]
Under this identification, $\widehat{M}$ is the $\widehat{R}$-linear map on $\widehat{Q}$ given by multiplication by $f_n(t+\rho, \mathbf{x}')$. Since $\widehat{M} = M \otimes_{R} \widehat{R}$, the determinant of $\widehat{M}$ is the image of $\det{M}$ under the natural embedding $R \hookrightarrow \widehat{R}$, which preserves the $t$-adic valuation. Hence
\begin{equation} \label{detcompletion}
v_t(\det M) = v_t(\det \widehat{M}).
\end{equation}

The special fiber is
\[
\begin{aligned}
Q_0:= \widehat{Q}/t\widehat{Q}
&\cong Q/tQ \\
&\cong \frac{R[x_1, \ldots, x_{n-1}]}
{\langle f_{1}, \ldots, f_{n-1}, t \rangle} \\
&\cong \frac{(R/(t))[x_1, \ldots, x_{n-1}]}{\langle \overline{f}_{1}, \ldots, \overline{f}_{n-1} \rangle} \\
&\cong \frac{\mathbb{C}[x_1, \ldots, x_{n-1}]}
{\langle f_1(\rho, \mathbf{x}'), \ldots, f_{n-1}(\rho, \mathbf{x}') \rangle}.
\end{aligned}
\]
Here, $\overline{f}_i$ denotes the image of the coefficients in the polynomial $f_i(\lambda, \mathbf{x}')$ under the map to $R/(t) \cong \mathbb{C}$. This is equivalent to directly replacing $\lambda$ with $\rho$. The maximal ideals of $Q_{0}$ correspond to the points of
\[
Y_{\rho} = \left\{\mathbf{q} \in \mathbb{C}^{n-1}: f_1(\rho, \mathbf{q}) = \cdots = f_{n-1}(\rho, \mathbf{q}) = 0 \right\}.
\]
Since $r(\rho) \neq 0$, Theorem~\ref{Poisson} implies that $Q_0$ is a finite-dimensional commutative Artinian algebra. In particular, $Y_{\rho}$ is finite. By the structure theorem \cite{AM69, Fan26} for finite-dimensional commutative algebras, we have
\[
Q_0 \cong \prod_{\mathbf{q} \in Y_{\rho}} Q_{\mathbf{q}},
\]
where $Q_{\mathbf{q}}$ is the local component obtained by localizing $Q_{0}$ at the point $\mathbf{q}$.

Let $e_q \in Q_0$ denote the idempotent corresponding to the factor $Q_q$. Since $\widehat{Q}$ is finite free over $\widehat{R} = \mathbb{C}[[t]]$, it is a finite $\widehat{R}$-algebra. Moreover, $\widehat{R}$ is a commutative Noetherian ring that is complete with respect to the ideal $(t)$, and the natural identification of special fibers gives
\[
Q_0 \cong \widehat{Q}/t\widehat{Q}.
\]
Thus, the hypotheses of Theorem~\ref{thm:idempotent-lifting} are satisfied. Applying this theorem, we obtain pairwise orthogonal idempotents $\widehat{e}_q \in \widehat{Q}$, for $q \in Y_{\rho}$, such that each $\widehat{e}_q$ lifts $e_q$. Furthermore, since $\widehat{Q}$ is commutative, each idempotent has a unique lift.

It follows from the pairwise orthogonality of the $\widehat{e}_q$ that their sum is an idempotent in $\widehat{Q}$. Its reduction modulo $t$ is
\[
\sum_{q \in Y_{\rho}} e_q = 1
\]
in $Q_0$. Hence, both $\sum_{q \in Y_{\rho}}\widehat{e}_q$ and $1$ are idempotent lifts of the same idempotent $1 \in Q_0$. By the uniqueness of idempotent lifts, we conclude that
\[
\sum_{q \in Y_{\rho}}\widehat{e}_q = 1.
\]
Consequently,
\[
\widehat{Q} \cong \prod_{q \in Y_{\rho}} \widehat{Q}_q, \quad \widehat{Q}_q:= \widehat{e}_{q} \widehat{Q},
\]
and $\widehat{Q}_q / t \widehat{Q}_q \cong Q_q$. Moreover, each $\widehat{Q}_q = \widehat{e}_q \widehat{Q}$ is a direct summand of the finite free $\widehat{R}$-module $\widehat{Q}$. Hence $\widehat{Q}_q$ is a finitely generated projective $\widehat{R}$-module. Since $\widehat{R}$ is a local ring, every finitely generated projective $\widehat{R}$-module is free. Thus, using $\widehat{Q}_q/t\widehat{Q}_q \cong Q_q$, we obtain
\[
\operatorname{rank}_{\widehat{R}} \widehat{Q}_q = \dim_{\mathbb{C}}(\widehat{Q}_q/t\widehat{Q}_q) = \dim_{\mathbb{C}} Q_q.
\]

Multiplication by $f_n$ preserves this product decomposition. Let $\widehat{M}_{\mathbf{q}}$ denote its restriction to $\widehat{Q}_{\mathbf{q}}$. Then
\[
\det \widehat{M} = \prod_{\mathbf{q} \in Y_{\rho}} \det \widehat{M}_{\mathbf{q}}, \quad v_t(\det \widehat{M}) = \sum_{\mathbf{q} \in Y_{\rho}} v_t(\det \widehat{M}_{\mathbf{q}}).
\]

We now calculate these contributions separately.

Suppose that $f_n(\rho, \mathbf{q}) \neq 0$. The image of $f_n$ in the residue field of the local Artinian algebra $Q_{\mathbf{q}}$ is $f_n(\rho, \mathbf{q})$, which is nonzero. Therefore, the image of $f_n$ in $Q_{\mathbf{q}}$ is a unit. It follows that the reduction of $\widehat{M}_{\mathbf{q}}$ modulo $t$ is invertible. Hence $\det \widehat{M}_{\mathbf{q}}$ is a unit in
$\widehat{R}$, and
\begin{equation} \label{zerocontribution}
v_t(\det \widehat{M}_{\mathbf{q}}) = 0.
\end{equation}

Suppose that $f_n(\rho, \mathbf{q}) = 0$. Set
\[
\mathbf{p}' = \mathbf{q}, \quad \mathbf{p} = (q_1, \ldots, q_{n-1}, 1)^{\rm T}.
\]
Then $\mathbf{p}$ represents a point of $\mathbb{V}_{\rho}(\mathcal{A}(H))$. By \eqref{detdelta}, $\Delta_{\mathbf{p}}$ is nonsingular.

Consider the local ring
\[
S = \mathbb{C}[x_1, \ldots, x_{n-1}]_{\langle x_1-q_1, \ldots, x_{n-1}-q_{n-1} \rangle}, \quad \mathfrak{m}_{\mathbf{q}} = \langle x_1-q_1, \ldots, x_{n-1}-q_{n-1} \rangle S,
\]
and let
\[
I = \langle f_1(\rho, \mathbf{x}'), \ldots, f_{n-1}(\rho, \mathbf{x}') \rangle S.
\]
The nonsingularity of $\Delta_{\mathbf{p}}$ implies that the linear parts of $f_1(\rho, \mathbf{x}'), \ldots, f_{n-1}(\rho, \mathbf{x}')$ form a basis of $\mathfrak{m}_{\mathbf{q}}/\mathfrak{m}_{\mathbf{q}}^2$.
Thus
\[
\mathfrak{m}_{\mathbf{q}} = I + \mathfrak{m}_{\mathbf{q}}^2.
\]
Applying Nakayama's lemma \cite{AM69} to the finitely generated $S$-module $\mathfrak{m}_{\mathbf{q}}/I$, we obtain that $\mathfrak{m}_{\mathbf{q}}/I=0$, and hence
\[
I = \mathfrak{m}_{\mathbf q}.
\]
Therefore, the local component corresponding to $\mathbf{q}$ is
\[
Q_{\mathbf{q}} \cong S/I \cong \mathbb{C}.
\]
It follows that $\widehat{Q}_{\mathbf{q}}$ has rank one over $\widehat{R}$.

Consider the natural $\widehat{R}$-algebra homomorphism
\[
\widehat{R} \rightarrow \widehat{Q}_{\mathbf{q}}, \quad a \mapsto a \, 1_{\widehat{Q}_{\mathbf{q}}}.
\]
Its reduction modulo $t$ is an isomorphism, since both reduced algebras are isomorphic to $\mathbb{C}$. By Nakayama's lemma \cite{AM69}, this homomorphism is surjective. Since both modules
are free of rank one over $\widehat{R}$, it is an isomorphism.

Under this isomorphism, the images of $x_1, \ldots, x_{n-1}$ determine formal power series $\mathbf{x}_{\mathbf{p}}(\rho+t)$ satisfying
\[
f_i(\rho+t, \mathbf{x}_{\mathbf{p}}(\rho+t)) = 0, \quad i \in [n-1]
\]
and
$\mathbf{x}_{\mathbf{p}}(\rho) = \mathbf{p}'$. Since $\Delta_{\mathbf{p}}$ is nonsingular, the formal implicit function theorem shows that these power series are precisely the
Taylor expansions of the unique local branch $\mathbf{x}_{\mathbf{p}}(\lambda)$ constructed earlier. Consequently, the image of $f_n$ under the above algebra isomorphism is
\[
g_{\mathbf{p}}(\rho+t) =  f_n(\rho+t, \mathbf{x}_{\mathbf{p}}(\rho+t)).
\]
Thus multiplication by $f_n$ on the rank-one $\widehat{R}$-algebra $\widehat{Q}_{\mathbf{q}}$ is multiplication by the scalar $g_{\mathbf{p}}(\rho+t)$. In particular,
\[
\det \widehat{M}_{\mathbf{q}} = g_{\mathbf{p}}(\rho+t).
\]
Since $g_{\mathbf{p}}(\lambda)$ has a simple zero at $\lambda = \rho$,
we obtain that
\begin{equation} \label{onecontribution}
v_t(\det \widehat{M}_{\mathbf{q}}) = 1.
\end{equation}

Finally, every point of $\mathbb{V}_{\rho}(\mathcal{A}(H))$ has all coordinates nonzero. Normalizing the last coordinate to $1$ therefore gives a bijection between $\mathbb{V}_{\rho}(\mathcal{A}(H))$ and
\[
Z_{\rho} = \left\{\mathbf{q} \in \mathbb{C}^{n-1}: f_1(\rho, \mathbf{q}) = \cdots = f_n(\rho, \mathbf{q}) = 0 \right\}.
\]
By \eqref{zerocontribution} and \eqref{onecontribution}, each point of $Z_{\rho}$ contributes exactly one to $v_t(\det \widehat{M})$, whereas every point of $Y_{\rho} \setminus Z_{\rho}$ contributes zero. Therefore
\[
\begin{aligned}
v_t(\det M)
&= v_t(\det \widehat{M}) \\
&= |Z_{\rho}| \\
&= \left| \mathbb{V}_{\rho}(\mathcal{A}(H)) \right|.
\end{aligned}
\]
Combining this equality with \eqref{amequal} proves the theorem.
\end{proof}

\section{Proof of Conjecture \ref{conj2}}

Let $H$ be a connected $k$-uniform hypergraph on $[n]$. Since $\mathcal{L}(H) = \mathcal{D}(H) - \mathcal{A}(H)$, we have
\[
(\mathcal{L}(H) \mathbf{x}^{k-1})_i = d_i x_{i}^{k-1} - \sum_{e \in E_i(H)} \mathbf{x}^{e \setminus \{i\}}, \quad i \in [n].
\]
In this section, let
\[
F_i(\lambda, \mathbf{x}) = (\lambda-d_i) x_i^{k-1} + \sum_{e \in E_i} x_{e \setminus \{i\}}, \quad i \in [n].
\]

In \cite{HQ14}, Hu and Qi characterized the eigenvectors associated with the zero eigenvalues of the Laplacian tensor of a uniform hypergraph.

\begin{theorem} [\cite{HQ14}] \label{zeroeigvec}
Let $H$ be a $k$-uniform connected hypergraph on $[n]$. A nonzero vector $\mathbf{x}$ is an eigenvector of the Laplacian tensor $\mathcal{L}(H)$ corresponding to the zero eigenvalue if and only if there exist nonzero $\gamma$ and integers $\alpha_{i}$ such that $x_{i} = \gamma \exp \left(\frac{2 \alpha_{i} \pi}{k} \sqrt{-1}\right)$ for $i \in [n]$, and
\[
\sum_{j \in e} \alpha_{j} = \sigma_{e} k, \quad \forall e \in E(H)
\]
for some integer $\sigma_{e}$ with $e \in E(H)$.
\end{theorem}

It was proved in \cite{Zheng24} that for a connected uniform hypergraph $H$ on $n$ vertices, the system of polynomials $\overline{F}_1(\lambda, \mathbf{x}), \dots, \overline{F}_{n-1}(\lambda, \mathbf{x})$ has only the trivial solution at $\lambda = 0$. Equivalently, this implies that the resultant ${\rm Res}(\overline{F}_1(0, \mathbf{x}), \dots, \overline{F}_{n-1}(0, \mathbf{x})) \neq 0$. In a related geometric direction, Fan \cite[Corollary 5.5]{Fan26} explored the local intersection multiplicity of the eigenvectors within the projective variety $\mathbb{V}_{0}(\mathcal{L}(H))$.
\begin{theorem} [\cite{Fan26}] \label{lsimple}
Let $H$ be a connected $k$-uniform hypergraph.
Then each point $p \in \mathbb{V}_{0}(\mathcal{L}(H))$ has local intersection multiplicity $1$.
\end{theorem}

We now prove Conjecture \ref{conj2} by adapting the argument in Section \ref{sec3}.

\begin{theorem}
Let $H$ be a connected $k$-uniform hypergraph with spectral radius $\rho$. Then
\[
{\rm am}(0, \mathcal{L}(H)) = {\rm am}(\rho, \mathcal{A}(H)).
\]
\end{theorem}
\begin{proof}
Let $\mathbf{x}' = (x_1, \ldots, x_{n-1})^{\rm T}$. For $i \in [n]$, denote
\[
\overline{F}_i(\lambda, \mathbf{x}') = F_i(\lambda, x_1, \ldots, x_{n-1}, 0), \quad f_i(\lambda, \mathbf{x}') = F_i(\lambda, x_1, \ldots, x_{n-1}, 1).
\]
Let $r(\lambda) = \operatorname{Res}(\overline{F}_1, \ldots, \overline{F}_{n-1})$, by \cite[Lemma 3.1]{Zheng24}, $r(0)\neq 0$.

Fix a projective point $P \in \mathbb{V}_0(\mathcal{L}(H))$. By Theorem~\ref{zeroeigvec} all coordinates of $P$ are nonzero. Thus $P$ has a unique representative
\[
\mathbf{p} = (p_1, \ldots, p_n)^{\rm T} ~~\text{with}~~ p_n = 1.
\]
Moreover, we have
\begin{equation} \label{sumvecpow}
\sum_{i=1}^n p_i^k = n.
\end{equation}

We now study the Jacobian matrix
\[
J(\mathbf{p}) = \left(\frac{\partial F_i}{\partial x_j}(0,\mathbf{p}) \right)_{n \times n}.
\]
Obviously, $J(\mathbf{p})$ is symmetric.

Let $\mathbf{p}' = (p_1, \ldots, p_{n-1})^{\rm T}$. On the affine chart $x_n = 1$, the point $P$ corresponds to $\mathbf{p}'$, which is a common zero of $f_1(0, \mathbf{x}'), \ldots, f_n(0, \mathbf{x}')$. Since $p_{n} \neq 0$, dehomogenization on the affine chart $x_{n} \neq 0$ preserves the local intersection multiplicity. By Theorem~\ref{lsimple}, its local intersection multiplicity is one. Hence $\mathbf{p}'$ is a simple zero, and Theorem~\ref{Jac} gives
\[
D_{\mathbf{p}} = \left(\frac{\partial f_i}{\partial x_j}(0, \mathbf{p}')\right)_{n \times (n-1)}, \quad \operatorname{rank}D_{\mathbf{p}} = n-1.
\]
In particular, this matrix consists precisely of the first $n-1$ columns of $J(\mathbf{p})$.

On the other hand, since each $F_i(0, \mathbf{x})$ is homogeneous of degree $k-1$, Euler's identity gives
\[
J(\mathbf{p})\mathbf{p} = (k-1)
\begin{pmatrix}
F_1(0, \mathbf{p}) \\
\vdots \\
F_n(0, \mathbf{p})
\end{pmatrix}
= \mathbf{0}.
\]
It follows that
\[
\operatorname{rank}J(\mathbf{p}) = n-1, \quad \ker J(\mathbf{p}) = \operatorname{span}\{\mathbf{p}\}.
\]
Note that $J(\mathbf{p})$ is symmetric and its adjugate has rank one. The column space and row space of the adjugate are spanned by $\mathbf{p}$ and $\mathbf{p}^{\rm T}$, respectively. Therefore, for some $c_{\mathbf{p}} \neq 0$,
\begin{equation} \label{ladjmatrix}
\operatorname{adj}J(\mathbf{p}) = c_{\mathbf{p}}\mathbf{p} \mathbf{p}^{\rm T}.
\end{equation}

Define
\[
\Delta_{\mathbf{p}} = \left(
\frac{\partial f_i}{\partial x_j} (0, \mathbf{p}')
\right)_{(n-1) \times (n-1)}.
\]
This is the leading principal $(n-1) \times (n-1)$ submatrix of $J(\mathbf{p})$. Hence,
\begin{equation} \label{ldetdelta}
\det \Delta_{\mathbf{p}} = \operatorname{adj}J(\mathbf{p})_{nn} = c_{\mathbf{p}}p_n^2 = c_{\mathbf{p}} \neq 0.
\end{equation}

By the holomorphic implicit function theorem, there exist neighborhoods $U$ of $0$ and $W$ of $\mathbf{p}'$, and a unique holomorphic map $\mathbf{x}_{\mathbf{p}}: U \rightarrow W$ such that
\[
f_i(\lambda, \mathbf{x}_{\mathbf{p}}(\lambda)) = 0, \quad i \in [n-1],
\]
and $\mathbf{x}_{\mathbf{p}}(0) = \mathbf{p}'$. Let
\[
g_{\mathbf{p}}(\lambda) = f_n(\lambda, \mathbf{x}_{\mathbf{p}}(\lambda)),
\]
then we claim that $g_{\mathbf{p}}(\lambda)$ has a simple zero at $\lambda = 0$.

Consider the augmented Jacobian matrix
\[
J_{\mathrm{aug}}(\mathbf{p}) = \left. \frac{\partial(f_1, \ldots, f_n)}{\partial(x_1, \ldots, x_{n-1}, \lambda)} \right|_{(0, \mathbf{p}')}.
\]
Since
\[
\frac{\partial f_i}{\partial \lambda}(0,\mathbf{p}') = \frac{\partial F_i}{\partial \lambda}(0, \mathbf{p}) = p_i^{k-1}, \quad i \in [n],
\]
then the last column of $J_{\mathrm{aug}}(\mathbf{p})$ is $(p_1^{k-1}, \ldots, p_n^{k-1})^{\rm T}$. Expanding its determinant along this column and using \eqref{ladjmatrix}, we obtain
\[
\begin{aligned}
\det J_{\mathrm{aug}}(\mathbf{p})
&= \sum_{i=1}^n \operatorname{adj}J(\mathbf{p})_{ni} p_i^{k-1} \\
&= c_{\mathbf{p}} p_n \sum_{i=1}^n p_i^k \\
&= nc_{\mathbf{p}},
\end{aligned}
\]
where we have used $p_n = 1$.

Let
\[
\mathbf{a} =\left(\frac{\partial f_n}{\partial x_1}(0, \mathbf{p}'), \ldots, \frac{\partial f_n}{\partial x_{n-1}}(0, \mathbf{p}') \right)^{\rm T},
\quad
\mathbf{b} = \left(\frac{\partial f_1}{\partial \lambda}(0, \mathbf{p}'), \ldots, \frac{\partial f_{n-1}}{\partial \lambda}(0, \mathbf{p}') \right)^{\rm T},
\]
and $d = \frac{\partial f_n}{\partial \lambda}(0, \mathbf{p}')$. Then the augmented Jacobian matrix has the following block form
\[
J_{\mathrm{aug}}(\mathbf{p}) =
\begin{pmatrix}
\Delta_{\mathbf{p}} & \mathbf{b} \\
\mathbf{a}^{\rm T} & d
\end{pmatrix}.
\]
Differentiating the first $n-1$ equations
\[
f_i(\lambda, \mathbf{x}_{\mathbf{p}}(\lambda)) = 0, \quad i \in [n-1]
\]
with respect to $\lambda$ at $\lambda = 0$, we obtain that
\[
\Delta_{\mathbf{p}} \dot{\mathbf{x}}_{\mathbf{p}} + \mathbf{b} = \mathbf{0}, \quad \dot{\mathbf{x}}_{\mathbf{p}} := \left. \frac{\mathrm{d} \mathbf{x}_{\mathbf{p}}(\lambda)}{\mathrm{d} \lambda} \right|_{\lambda = 0}.
\]
Since $\Delta_{\mathbf{p}}$ is nonsingular, it follows that $\dot{\mathbf{x}}_{\mathbf{p}} = -\Delta_{\mathbf{p}}^{-1}\mathbf{b}$.

Therefore, by the Schur complement formula and \eqref{ldetdelta},
\begin{equation} \label{lpartequ}
\begin{aligned}
g_{\mathbf{p}}'(0)
&= d + \mathbf{a}^{\rm T} \dot{\mathbf{x}}_{\mathbf{p}} \\
&= d - \mathbf{a}^{\rm T} \Delta_{\mathbf{p}}^{-1} \mathbf{b} \\
&= \frac{\det J_{\mathrm{aug}}(\mathbf{p})} {\det \Delta_{\mathbf{p}}} = n \neq 0.
\end{aligned}
\end{equation}
Thus $g_{\mathbf{p}}(\lambda)$ has a simple zero at $\lambda = 0$.

Let
\[
R = \mathbb{C}[\lambda]_{(\lambda)} = \left\{\frac{f(\lambda)}{g(\lambda)}: f, g \in \mathbb{C}[\lambda],\ g(0) \neq 0 \right\}
\]
and
\[
Q = R[x_1, \ldots, x_{n-1}]/\langle f_1, \ldots, f_{n-1} \rangle.
\]
Then $R$ is a discrete valuation ring with uniformizer $t = \lambda$, maximal ideal $(t)$, residue field $R/(t)\cong\mathbb{C}$, and fraction field $\mathbb{C}(\lambda)$. Its normalized valuation $v_t$ is the order of vanishing at $\lambda = 0$.

Let $M: Q \rightarrow Q$ be the $R$-linear endomorphism given by multiplication by $f_n$. Since $r(0) \neq 0$, $r(\lambda)$ is a unit in $R$. Moreover, $\operatorname{Res}(F_1, \ldots, F_n) = \varphi_{\mathcal{L}(H)}(\lambda)$ is a nonzero polynomial. Thus the hypotheses of the local Poisson formula \cite[Lemma 3.2]{Fan26} are satisfied, and we obtain that
\[
\operatorname{am}(0, \mathcal{L}(H)) = v_t(\det M).
\]

We now compute $v_t(\det M)$. Let
\[
\widehat{R} = \mathbb{C}[[t]], \quad \widehat{Q} = Q \otimes_{R} \widehat{R}, \quad \widehat{M} = M \otimes_{R} \widehat{R}.
\]
Here the embedding $R \hookrightarrow \widehat{R}$ is given by expanding each element of $R$ as a formal power series in $t = \lambda$. Under this embedding,
\[
\widehat{Q} \cong \frac{\mathbb{C}[[t]][x_1, \ldots, x_{n-1}]}{\langle f_1(t, \mathbf{x}'), \ldots, f_{n-1}(t, \mathbf{x}') \rangle},
\]
and $\widehat{M}$ is the $\widehat{R}$-linear map on $\widehat{Q}$ given by multiplication by $f_n(t, \mathbf{x}')$. Since $\widehat{M} = M \otimes_{R} \widehat{R}$, the determinant of $\widehat{M}$ is the image of $\det{M}$ under the natural embedding $R \hookrightarrow \widehat{R}$, which preserves the $t$-adic valuation. Hence
\begin{equation} \label{ldetcompletion}
v_t(\det M) = v_t(\det \widehat{M}).
\end{equation}

The special fiber is
\[
\begin{aligned}
Q_0 := \widehat{Q}/t\widehat{Q}
&\cong Q/tQ \\
&\cong \frac{\mathbb{C}[x_1, \ldots, x_{n-1}]}{\langle f_1(0, \mathbf{x}'), \ldots, f_{n-1}(0, \mathbf{x}') \rangle}.
\end{aligned}
\]
Its maximal ideals correspond to the points of
\[
Y_0 = \left\{\mathbf{q} \in \mathbb{C}^{n-1}: f_1(0, \mathbf{q}) = \cdots = f_{n-1}(0, \mathbf{q}) = 0 \right\}.
\]
Since $r(0) \neq 0$, Theorem~\ref{Poisson} implies that $Q_0$ is a finite-dimensional commutative Artinian algebra. In particular, $Y_0$ is finite. By the structure theorem \cite{AM69, Fan26} for finite-dimensional commutative algebras, we have
\[
Q_0 \cong \prod_{\mathbf{q} \in Y_{0}} Q_{\mathbf{q}},
\]
where $Q_{\mathbf{q}}$ is the local component obtained by localizing $Q_{0}$ at the point $\mathbf{q}$.

Let $e_q \in Q_0$ denote the idempotent corresponding to the factor $Q_q$. Then by Theorem~\ref{thm:idempotent-lifting}, the idempotents $e_q$ lift to pairwise orthogonal idempotents $\widehat{e}_q \in \widehat{Q}$ for $q \in Y_0$. Here, we use the facts that $\widehat{Q}$ is finite free over $\widehat{R} = \mathbb{C}[[t]]$ and that $Q_0 \cong \widehat{Q}/t\widehat{Q}$. Moreover, since $\widehat{Q}$ is commutative, each lift is unique. Thus, $\sum_{q \in Y_0} \widehat{e}_q$ is an idempotent whose reduction modulo $t$ is
\[
\sum_{q \in Y_0} e_q = 1
\]
in $Q_0$. By the uniqueness of idempotent lifts, we obtain
\[
\sum_{q \in Y_{0}} \widehat{e}_q = 1.
\]
Consequently,
\[
\widehat{Q} \cong \prod_{q \in Y_{0}} \widehat{Q}_q, \quad \widehat{Q}_q:= \widehat{e}_{q} \widehat{Q},
\]
and $\widehat{Q}_q / t \widehat{Q}_q \cong Q_q$. Moreover, each $\widehat{Q}_q = \widehat{e}_q \widehat{Q}$ is a direct summand of the finite free $\widehat{R}$-module $\widehat{Q}$. Hence $\widehat{Q}_q$ is a finitely generated projective $\widehat{R}$-module. Since $\widehat{R}$ is a local ring, every finitely generated projective $\widehat{R}$-module is free. Thus, using
$\widehat{Q}_q/t\widehat{Q}_q \cong Q_q$, we obtain
\[
\operatorname{rank}_{\widehat{R}} \widehat{Q}_q = \dim_{\mathbb{C}} (\widehat{Q}_q/t\widehat{Q}_q) = \dim_{\mathbb{C}} Q_q.
\]

Multiplication by $f_n$ preserves this product decomposition. Let $\widehat{M}_{\mathbf{q}}$ denote its restriction to $\widehat{Q}_{\mathbf{q}}$. Then
\[
\det \widehat{M} = \prod_{\mathbf{q} \in Y_0} \det \widehat{M}_{\mathbf{q}}, \quad v_t(\det \widehat{M}) = \sum_{\mathbf{q} \in Y_0} v_t(\det \widehat{M}_{\mathbf{q}}).
\]

We calculate these contributions separately.

Suppose that $f_n(0, \mathbf{q}) \neq 0$. The image of $f_n$ in the residue field of the local Artinian algebra $Q_{\mathbf{q}}$ is $f_n(0, \mathbf{q})$, which is nonzero. Therefore, the image of $f_n$ in $Q_{\mathbf{q}}$ is a unit. It follows that the reduction of $\widehat{M}_{\mathbf{q}}$ modulo $t$ is invertible. Hence $\det \widehat{M}_{\mathbf{q}}$ is a unit in $\widehat{R}$, and
\begin{equation} \label{lzerocontribution}
v_t(\det \widehat{M}_{\mathbf{q}}) = 0.
\end{equation}

Suppose that $f_n(0, \mathbf{q}) = 0$. Set
\[
\mathbf{p}' = \mathbf{q}, \quad \mathbf{p} = (q_1, \ldots, q_{n-1}, 1)^{\rm T}.
\]
Then $\mathbf{p}$ represents a point of $\mathbb{V}_0(\mathcal{L}(H))$. By \eqref{ldetdelta}, $\Delta_{\mathbf{p}}$ is nonsingular.

Consider the local ring
\[
S = \mathbb C[x_1,\ldots,x_{n-1}]_{\langle x_1-q_1, \ldots, x_{n-1}-q_{n-1}\rangle}, \quad \mathfrak{m}_{\mathbf{q}} = \langle x_1-q_1, \ldots, x_{n-1}-q_{n-1} \rangle S,
\]
and let
\[
I = \langle f_1(0, \mathbf{x}'), \ldots, f_{n-1}(0, \mathbf{x}') \rangle S.
\]
The nonsingularity of $\Delta_{\mathbf p}$ implies that the linear parts of $f_1(0, \mathbf{x}'), \ldots, f_{n-1}(0, \mathbf{x}')$ form a basis of $\mathfrak{m}_{\mathbf{q}}/\mathfrak{m}_{\mathbf{q}}^2$.
Thus
\[
\mathfrak{m}_{\mathbf{q}} = I + \mathfrak{m}_{\mathbf{q}}^2.
\]
Applying Nakayama's lemma to the finitely generated $S$-module $\mathfrak {m}_{\mathbf{q}}/I$, we obtain $\mathfrak{m}_{\mathbf{q}}/I=0$, and hence
\[
I = \mathfrak{m}_{\mathbf{q}}.
\]
Therefore, the local component corresponding to $\mathbf{q}$ is
\[
Q_{\mathbf{q}} \cong S/I \cong \mathbb{C}.
\]
It follows that $\widehat{Q}_{\mathbf{q}}$ has rank one over $\widehat{R}$.

Consider the natural $\widehat{R}$-algebra homomorphism
\[
\widehat{R} \rightarrow \widehat{Q}_{\mathbf{q}}, \quad a \mapsto a\, 1_{\widehat{Q}_{\mathbf{q}}}.
\]
Its reduction modulo $t$ is an isomorphism, since both reduced algebras are isomorphic to $\mathbb{C}$. By Nakayama's lemma, this homomorphism is surjective. Since both modules
are free of rank one over $\widehat{R}$, it is an isomorphism.

Under this isomorphism, the images of $x_1, \ldots, x_{n-1}$ determine formal power series $\mathbf{x}_{\mathbf{p}}(t)$ satisfying
\[
f_i(t, \mathbf{x}_{\mathbf{p}}(t)) = 0, \quad i \in [n-1]
\]
and $\mathbf{x}_{\mathbf{p}}(0) = \mathbf{p}'$. Since $\Delta_{\mathbf{p}}$ is nonsingular, the formal implicit function theorem shows that these power series are precisely the Taylor expansions of the unique local branch $\mathbf{x}_{\mathbf{p}}(\lambda)$ constructed earlier. Consequently, the image of $f_n$ under the above algebra isomorphism is
\[
g_{\mathbf{p}}(t) = f_n(t, \mathbf{x}_{\mathbf{p}}(t)).
\]
Thus multiplication by $f_n$ on the rank-one $\widehat{R}$-algebra $\widehat{Q}_{\mathbf{q}}$ is multiplication by the scalar $g_{\mathbf{p}}(t)$. In particular,
\[
\det \widehat{M}_{\mathbf{q}} = g_{\mathbf{p}}(t).
\]
By \eqref{lpartequ}, $g_{\mathbf{p}}(\lambda)$ has a simple zero at $\lambda = 0$. Therefore,
\begin{equation} \label{lonecontribution}
v_t(\det \widehat{M}_{\mathbf{q}}) = 1.
\end{equation}

Finally, since every point of $\mathbb{V}_0(\mathcal{L}(H))$ has a representative with last coordinate equal to $1$, normalization gives a bijection between
$\mathbb{V}_0(\mathcal{L}(H))$ and
\[
Z_0 = \left\{\mathbf{q} \in \mathbb{C}^{n-1}: f_1(0, \mathbf{q}) = \cdots = f_n(0, \mathbf{q}) = 0 \right\}.
\]
By \eqref{lzerocontribution} and \eqref{lonecontribution}, we obtain
\[
\begin{aligned}
\operatorname{am}(0, \mathcal{L}(H))
&= v_t(\det M) \\
&= v_t(\det \widehat{M}) \\
&= |Z_0| = |\mathbb{V}_0(\mathcal{L}(H))|.
\end{aligned}
\]
By \cite[Eq.(4.2)]{Fan26}, the equality
\[
|\mathbb{V}_0(\mathcal{L}(H))| = |\mathbb{V}_{\rho}(\mathcal{A}(H))|
\]
holds for any connected $k$-uniform hypergraph $H$. Combining it with Theorem~\ref{thm1}, we have
\[
\operatorname{am}(0, \mathcal{L}(H)) = \operatorname{am}(\rho, \mathcal{A}(H)).
\]
This completes the proof.
\end{proof}

\section{Acknowledgements}

This research was done with the aid of AI under the author's guidance. In particular, discussions with the AI helped clarify how idempotent lifting extends the local-factor decomposition of the Artinian special fiber to the completed algebra.


\bibliographystyle{amsplain}

\begin{thebibliography}{99}

\bibitem{AM69}
Michael F. Atiyah and Ian G. Macdonald,
\emph{Introduction to Commutative Algebra},
Addison-Wesley, Reading, MA, 1969.

\bibitem{CD12}
Joshua Cooper and Aaron Dutle,
\emph{Spectra of uniform hypergraphs},
Linear Algebra Appl. \textbf{436} (2012), 3268--3292.

\bibitem{CLO97}
David Cox, John Little, and Donal O'Shea,
\emph{Ideals, Varieties, and Algorithms: An Introduction to Computational Algebraic Geometry and Commutative Algebra},
Undergraduate Texts in Mathematics,
Springer-Verlag, New York, 1997.

\bibitem{CLO98}
David Cox, John Little, and Donal O'Shea,
\emph{Using Algebraic Geometry},
Graduate Texts in Mathematics,
Springer-Verlag, New York, 1998.

\bibitem{Eisenbud95}
David Eisenbud,
\emph{Commutative Algebra with a View Toward Algebraic Geometry},
Graduate Texts in Mathematics,
Springer-Verlag, New York, 1995.

\bibitem{Fan26}
Yi-Zheng Fan,
\emph{The multiplicity of eigenvalues of nonnegative tensors and hypergraphs},
\href{https://doi.org/10.48550/arXiv.2410.20830}{arXiv:2410.20830v3}, 2026.

\bibitem{FBH19}
Yi-Zheng Fan, Yan-Hong Bao, and Tao Huang,
\emph{Eigenvariety of nonnegative symmetric weakly irreducible tensors associated with spectral radius and its application to hypergraphs},
Linear Algebra Appl. \textbf{564} (2019), 72--94.

\bibitem{FHB22}
Yi-Zheng Fan, Tao Huang, and Yan-Hong Bao,
\emph{The dimension of eigenvariety of nonnegative tensors associated with spectral radius},
Proc. Amer. Math. Soc. \textbf{150} (2022), 2287--2299.

\bibitem{FTL22}
Yi-Zheng Fan, Meng-Yu Tian, and Min Li,
\emph{The stabilizing index and cyclic index of the coalescence and Cartesian product of uniform hypergraphs},
J. Combin. Theory Ser. A \textbf{185} (2022), 105537.

\bibitem{FWBWLZ19}
Yi-Zheng Fan, Yi Wang, Yan-Hong Bao, Jiang-Chao Wan, Min Li, and Zhu Zhu,
\emph{Eigenvectors of Laplacian or signless Laplacian of hypergraphs associated with zero eigenvalue},
Linear Algebra Appl. \textbf{579} (2019), 244--261.

\bibitem{FGH13}
S. Friedland, S. Gaubert, and L. Han,
\emph{Perron--Frobenius theorem for nonnegative multilinear forms and extensions},
Linear Algebra Appl. \textbf{438} (2013), 738--749.

\bibitem{FG02}
Klaus Fritzsche and Hans Grauert,
\emph{From Holomorphic Functions to Complex Manifolds},
Graduate Texts in Mathematics,
Springer-Verlag, New York, 2002.

\bibitem{GKZ94}
I. M. Gelfand, M. M. Kapranov, and A. V. Zelevinsky,
\emph{Discriminants, Resultants and Multidimensional Determinants},
Birkh{\"a}user Boston, Inc., Boston, MA, 1994.

\bibitem{Hazewinkel78}
Michiel Hazewinkel,
\emph{Formal Groups and Applications},
Pure and Applied Mathematics,
Academic Press, New York, 1978.

\bibitem{HQ14}
Shenglong Hu and Liqun Qi,
\emph{The eigenvectors associated with the zero eigenvalues of the Laplacian and signless Laplacian tensors of a uniform hypergraph},
Discrete Appl. Math. \textbf{169} (2014), 140--151.

\bibitem{HY16}
Shenglong Hu and Ke Ye,
\emph{Multiplicities of tensor eigenvalues},
Commun. Math. Sci. \textbf{14} (2016), 1049--1071.

\bibitem{KF15}
Murad-ul-Islam Khan and Yi-Zheng Fan,
\emph{On the spectral radius of a class of non-odd-bipartite even uniform hypergraphs},
Linear Algebra Appl. \textbf{480} (2015), 93--106.

\bibitem{Lim05}
Lek-Heng Lim,
\emph{Singular values and eigenvalues of tensors: a variational approach},
in: Proceedings of the 1st IEEE International Workshop on Computational Advances in Multi-Sensor Adaptive Processing,
vol. 1, pp. 129--132, 2005.

\bibitem{Qi05}
Liqun Qi,
\emph{Eigenvalues of a real supersymmetric tensor},
J. Symbolic Comput. \textbf{40} (2005), 1302--1324.

\bibitem{Qi14}
Liqun Qi,
\emph{H$^+$-eigenvalues of Laplacian and signless Laplacian tensors},
Commun. Math. Sci. \textbf{12} (2014), 1045--1064.

\bibitem{YY11}
Yuning Yang and Qingzhi Yang,
\emph{On some properties of nonnegative weakly irreducible tensors},
\href{https://doi.org/10.48550/arXiv.1111.0713}{arXiv:1111.0713}, 2011.

\bibitem{Zheng24}
Ya-Nan Zheng,
\emph{The zero eigenvalue of the Laplacian tensor of a uniform hypergraph},
Linear Multilinear Algebra \textbf{72} (2024), 1094--1111.

\end{thebibliography}

\end{document}